\documentclass{amsart}

\usepackage{amsmath,amssymb,amsfonts,enumerate,amsthm,graphicx,color}
\usepackage{amscd,amsmath,amsthm,amssymb}
\usepackage{pstcol,pst-plot,pst-3d}
\usepackage{lmodern,pst-node}
\usepackage{lineno}

\def\opn#1#2{\def#1{\operatorname{#2}}} 
\opn\chara{char} \opn\length{\ell} \opn\pd{pd} \opn\rk{rk}
\opn\projdim{proj\,dim} \opn\injdim{inj\,dim} \opn\rank{rank}
\opn\depth{depth} \opn\grade{grade} \opn\height{height}
\opn\embdim{emb\,dim} \opn\codim{codim}

\opn\Tr{Tr} \opn\bigrank{big\,rank}
\opn\superheight{superheight}\opn\lcm{lcm}
\opn\trdeg{tr\,deg}%
\opn\reg{reg} \opn\lreg{lreg} \opn\skel{skel}
\opn\multideg{multideg}
\opn\div{div} \opn\Div{Div} \opn\cl{cl} \opn\Cl{Cl}
\opn\Spec{Spec} \opn\Supp{Supp} \opn\supp{supp} \opn\Sing{Sing}
\opn\Ass{Ass}
\opn\Ann{Ann} \opn\Rad{Rad} \opn\Soc{Soc}
\opn\Ker{Ker} \opn\Coker{Coker} \opn\Im{Im} \opn\Hom{Hom}
\opn\Tor{Tor} \opn\Ext{Ext} \opn\End{End} \opn\Aut{Aut}
\opn\id{id}

\opn\nat{nat}
\opn\pff{pf}
\opn\Pf{Pf} \opn\GL{GL} \opn\SL{SL} \opn\mod{mod} \opn\ord{ord}
\opn\aff{aff} \opn\con{conv} \opn\relint{relint} \opn\st{st}
\opn\lk{lk} \opn\cn{cn} \opn\core{core} \opn\vol{vol}
\opn\link{link} \opn\star{star} \opn\skel{skel} \opn\Reg{Reg}
\opn\gr{gr}

\def\pot#1#2{#1[\kern-0.28ex[#2]\kern-0.28ex]}

\opn\dirlim{\underrightarrow{\lim}}
\opn\inivlim{\underleftarrow{\lim}}
\def\Implies{\ifmmode\Longrightarrow \else
     \unskip${}\Longrightarrow{}$\ignorespaces\fi}
\def\implies{\ifmmode\Rightarrow \else
     \unskip${}\Rightarrow{}$\ignorespaces\fi}
\def\iff{\ifmmode\Longleftrightarrow \else
     \unskip${}\Longleftrightarrow{}$\ignorespaces\fi}

\let\:=\colon

\newtheorem{thm}{Theorem}[section]
\newtheorem{cor}[thm]{Corollary}
\newtheorem{lem}[thm]{Lemma}
\newtheorem{prop}[thm]{Proposition}
\newtheorem{defn}[thm]{Definition}

\newtheorem{exam}[thm]{Example}
\newtheorem{rem}[thm]{Remark}

\numberwithin{equation}{section}

\begin{document}

\bibliographystyle{amsplain}

\title{ Binomial edge ideals with pure resolutions }
\author{ Dariush Kiani$^{1}$ and Sara Saeedi Madani$^{2}$ }
\thanks{2010 \textit{Mathematics Subject Classification.} 05E40, 05C25, 16E05, 13C05}
\thanks{\textit{Key words and phrases.} Binomial edge ideal, Pure resolution, Linear strand. }
\thanks{ $^{1,2}$Department of Pure Mathematics,
 Faculty of Mathematics and Computer Science,
 Amirkabir University of Technology (Tehran Polytechnic),
424, Hafez Ave., Tehran 15914, Iran, and School of Mathematics, Institute for Research in Fundamental Sciences (IPM),
P.O. Box 19395-5746, Tehran, Iran. }
\thanks{$^{1}$dkiani@aut.ac.ir, dkiani7@gmail.com, $^{2}$sarasaeedi@aut.ac.ir }

\begin{abstract}
We characterize all graphs whose binomial edge ideals have pure resolutions. Moreover, we introduce a new switching of graphs which does not change some algebraic invariants of graphs, and using this, we study the linear strand of the binomial edge ideals for some classes of graphs. Also, we pose a conjecture on the linear strand of such ideals for every graph.
\end{abstract}

\maketitle

\section{ Introduction }\label{Introduction}

\noindent The binomial edge ideal of a graph was introduced in \cite{HHHKR} and \cite{O}.
Let $G$ be a finite simple graph with vertex set $[n]$ and edge set $E(G)$ and let $S=K[x_1,\ldots,x_n,y_1,\ldots,y_n]$ be the polynomial ring over a field $K$. Then the \textbf{binomial edge ideal} of $G$ in $S$,
denoted by $J_G$, is generated by binomials $f_{ij}=x_iy_j-x_jy_i$, where $i<j$ and $\{i,j\}\in E(G)$. One could see this ideal as an ideal
generated by a collection of 2-minors of a $(2\times n)$-matrix whose entries are all indeterminates. Many of the algebraic properties of such ideals
were studied in \cite{D}, \cite{EHH},  \cite{EZ}, \cite{HHHKR}, \cite{KS}, \cite{MM}, \cite{SK} and \cite{Z}. In \cite{HHHKR}, the authors determined all graphs whose binomial edge ideal have a Gr\"{o}bner basis with respect to the lexicographic order induced by $x_1>\cdots >x_n>y_1>\cdots >y_n$, and called this class of graphs, \textbf{closed graphs}. There are also some combinatorial descriptions for these graphs. In \cite{EHHQ}, the authors introduced a generalization of the binomial edge ideal of a graph and some properties of this ideal were studied in \cite{EHHQ} and \cite{SK1}.

In this paper, we characterize all binomial edge ideals with pure resolutions. We also try to compute the linear strand of binomial edge ideals. We introduce a new method of switching for graphs which preserves many algebraic invariants of the binomial edge ideals. This paper is organized as follows. In Section~\ref{Binomial edge ideals with pure resolutions}, we classify all graphs such that their binomial edge ideals have pure resolutions, which generalizes the result of Schenzel and Zafar who studied complete bipartite graphs. In Section~\ref{Free cut edge switching and linear strand of binomial edge ideals}, we introduce a new switching method for graphs, where the graded Betti numbers, regularity and projective dimension of binomial edge ideals are invariant under it. Using that, we compute the graded Betti numbers of some classes of graphs, generalizing some of the results of Zahid and Zafar in \cite{ZZ}. Also, we compute the linear strand of binomial edge ideals of some classes of graphs and pose a formula with respect to graphical terms as a conjecture, for an arbitrary graph.

\section{Binomial edge ideals with pure resolutions }\label{Binomial edge ideals with pure resolutions}

\noindent In this section, we investigate about pure resolutions of binomial edge ideals. Let $I$ be a homogeneous ideal of $S$ whose generators
all have degree $d$. Then $I$ has a \textbf{$d$-pure resolution} (or pure resolution) if its minimal graded free resolution can be written in the form
$$0\rightarrow S(-d_p)^{\beta_{p}(I)}\rightarrow\cdots\rightarrow
S(-d_1)^{\beta_{1}(I)}\rightarrow I\rightarrow0,$$ where $d=d_1$.
In addition, we say that $I$ has a \textbf{$d$-linear resolution} (or linear resolution) if
for all $i\geq 0$, $\beta_{i,j}(I)=0$ for all $j\neq{i+d}$. In \cite{SK} and \cite{SK1}, all binomial edge ideals with linear resolution are characterized.

\begin{thm}\label{linear}
\cite[Theorem~2.1]{SK} Let $G$ be graph. Then $J_G$ has a linear resolution if and only if $G$ is a complete graph.
\end{thm}

A similar question could be asked about pure resolutions. Indeed, characterizing all binomial edge ideals which has pure resolutions with respect to combinatorial terms would be interesting too. On the other hand, pure resolutions are important from the Boij-S\"{o}derberg theory's point of view, in which pure resolutions can be used
as building blocks in order to obtain any Betti diagram. The following theorem which characterizes all graphs whose binomial edge ideals have pure resolutions, is the main theorem of this section.

\begin{thm}\label{pure}
Let $G$ be a graph without any isolated vertices. Then $J_G$ has a pure resolution if and only if $G$ is a \\
{\em{(a)}} complete graph, or \\
{\em{(b)}} complete bipartite graph, or \\
{\em{(c)}} disjoint union of some paths.
\end{thm}

To prove the above theorem, we need some facts which are mentioned in the following. Let $G$ and $H$ be two graphs. Then we say that $G$ is $H$-free, if it does not have any induced subgraphs isomorphic to the graph $H$. Here, by $C_n$, $K_n$, $P_n$ and $K_{m,n-m}$, for some $m$ with $1\leq m \leq n$, we mean the cycle, complete graph, path and complete bipartite graph, on $n$ vertices, respectively. Also, by $G\setminus v$, we mean the induced subgraph of a graph $G$ on $[n]\setminus v$.

\begin{prop}\label{induced}
\cite[Proposition~8]{SK1} Let $G$ be a graph and $H$ an induced subgraph of $G$. Then we have $\beta_{i,j}(J_{H})\leq \beta_{i,j}(J_{G})$, for all $i,j$. In particular, if $J_G$ has a pure resolution, then $J_H$ does too.
\end{prop}

\begin{rem}\label{special graphs}
{\em Suppose that we label the vertices of $P_n$ such that its edges are of the form $\{v_i,v_{i+1}\}$, for all $1\leq i\leq n-1$. Then, since $P_n$ is a closed graph, we have $\mathrm{in}_{<}(J_{P_n})$ is minimally generated by $\{x_iy_{i+1}:1\leq i\leq n-1\}$, which is a regular sequence of monomials. It follows that the generators of $J_{P_n}$ form a regular sequence as well. Therefore, the Koszul complex resolves $S/J_{P_n}$. Consequently, $S/J_{P_n}$ has a pure resolution, and $\beta_{i,j}(S/J_{P_n})\neq 0$, if $j=2i$, and $\beta_{i,j}(S/J_{P_n})=0$, otherwise. In particular, we have $\beta_{2,6}(J_{P_n})\neq 0$. Moreover, note that, by \cite[Theorem~5.4]{SZ}, $\beta_{2,5}(J_{K_{1,3}})\neq 0$. Thus, by Proposition~\ref{induced}, if $G$ is either the graph depicted in Figure~\ref{G_2} or one in Figure~\ref{square}, then its binomial edge ideal does not have a pure resolution, since both graphs have $P_4$ and $K_{1,3}$ as induced subgraphs and hence $\beta_{2,5}(J_{G})\neq 0$ and $\beta_{2,6}(J_{G})\neq 0$, by Proposition~\ref{induced}. }
\end{rem}

\begin{center}
\begin{figure}
\hspace{0 cm}
\includegraphics[height=1.5cm,width=3.1cm]{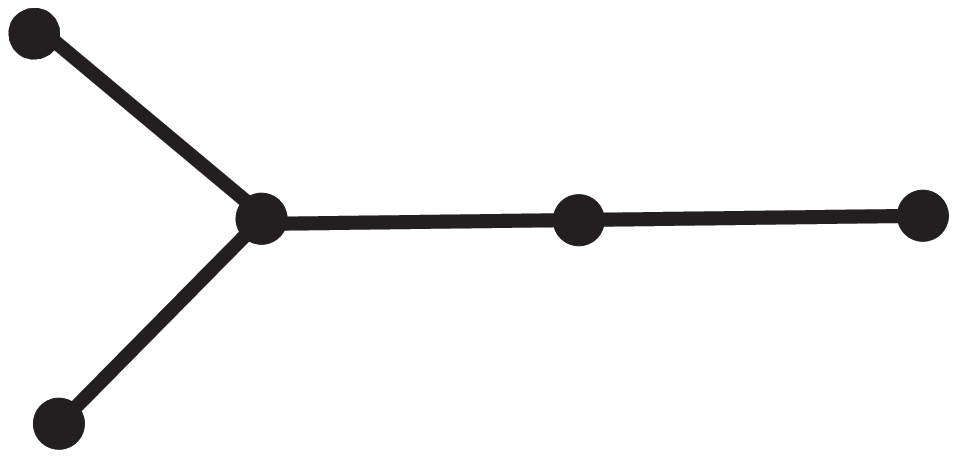}
\caption{\footnotesize{}}\hspace{2 cm}
\label{G_2}
\end{figure}
\end{center}

\begin{center}
\begin{figure}
\hspace{0 cm}
\includegraphics[height=1.85cm,width=1.8cm]{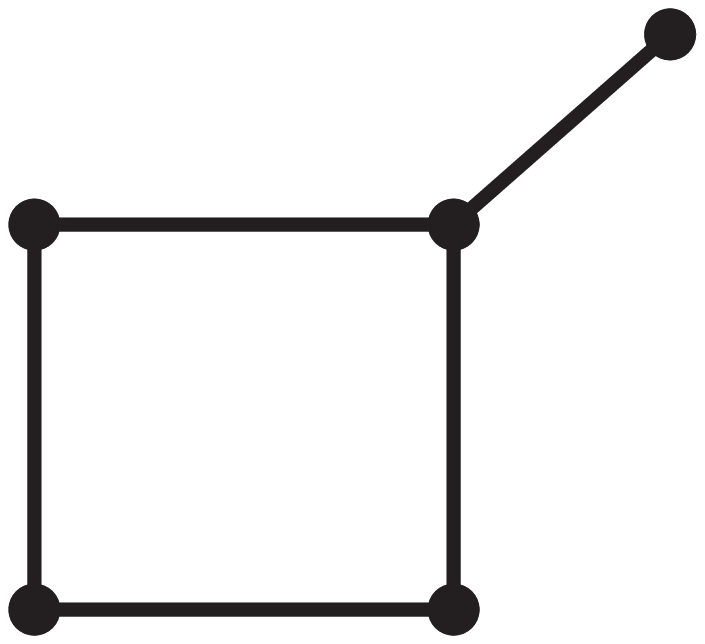}
\caption{\footnotesize{}}\hspace{2 cm}
\label{square}
\end{figure}
\end{center}

We need the following easy lemma to prove the main theorem.

\begin{lem}\label{lemma}
Let $I\subset R=K[x_1,\ldots,x_n]$ and $J\subset T=K[x_{n+1},\ldots,x_m]$ be two graded ideals with pure resolutions, say
$$0\rightarrow R(-d_p)^{\beta_p}\rightarrow \cdots \rightarrow R(-d_1)^{\beta_1}\rightarrow R \rightarrow R/I\rightarrow 0,$$
$$0\rightarrow T(-e_q)^{\gamma_q}\rightarrow \cdots \rightarrow T(-e_1)^{\beta_1}\rightarrow T \rightarrow T/J\rightarrow 0,$$
and let $S=K[x_1,\ldots,x_m]$. Then $IS+JS$ has a pure resolution if and only if $e_1=d_1$, $d_i=id_1$, and $e_j=je_1$, for all $i,j$.
\end{lem}

\begin{proof}
It is enough to note that the minimal graded free resolution of $S/(IS+JS)$ is the tensor product of those of $R/I$ and $T/J$. Thus, we have
$$\beta_{i,j}(S/\big{(}IS+JS)\big{)}=\sum_{t+t'=i,k+k'=j}\beta_{t,k}(R/I)\beta_{t',k'}(T/J),$$ for all $i,j$.
\end{proof}

{\em Proof of Theorem~\ref{pure}.} By Theorem~\ref{linear}, $J_{K_n}$ has a linear resolution, and hence a pure resolution. By Remark~\ref{special graphs}, $J_{P_n}$ has also a pure resolution. So that, if $G$ is the disjoint union of some paths, say $P_{t_1},\ldots,P_{t_l}$, then the minimal graded free resolution of $S/J_G$ is the tensor product of the minimal graded free resolutions of $S/J_{P_{t_j}}$'s, and hence it is clearly pure. If $G$ is a complete bipartite graph, then by \cite[Theorem~5.3]{SZ}, $J_{G}$ has a pure resolution. Conversely, suppose that $J_G$ has a pure resolution. If it also has a linear resolution, then by Theorem~\ref{linear}, $G$ is complete. So, suppose that $J_G$ has a pure and non-linear resolution. Thus, $G$ is a non-complete graph. First suppose that $G$ is a connected graph on $[n]$. So that $\beta_{1,4}(J_G)\neq 0$, by \cite[Theorem~2.2]{SK}. Hence, $\beta_{1,3}(J_G)=0$, since $J_G$ has a pure resolution. Thus, $G$ is a $C_3$-free graph, by \cite[Theorem~2.2]{SK}. If $G$ has an induced graph isomorphic to $C_m$, for $m\geq 5$, then $J_G$ does not have any pure resolutions, by Proposition~\ref{induced} and \cite[Corollary~3.8]{ZZ}. Thus, we have that $G$ is $C_m$-free, for all $m\neq 4$. In particular, $G$ does not contain any induced odd cycles and hence any odd cycles, which yields that $G$ is a bipartite graph. Let $A$ and $B$ be the bipartition of the vertex set of $G$. If $G$ is a tree, then it is $P_n$, or $K_{1,n-1}$, or has an induced subgraph isomorphic the graph shown in Figure~\ref{G_2}. But, in the latter case, $J_G$ does not have any pure resolution, by Remark~\ref{special graphs}. Now, suppose that $G$ has at least one induced cycle which has to be of length $4$. So, we have $|A|,|B|\geq 2$. We show that $G$ is complete bipartite. For this purpose, we use the induction on $n$, the number of vertices. If $n=4$, then $G$ is just $C_4$, which is a complete bipartite graph. Now, suppose that $n\geq 5$ and $|A|\leq |B|$. Let $C$ be an induced cycle of $G$ on the vertices $\{v,w,z,t\}$, where $v,w\in A$ and $z,t\in B$. Note that $G\setminus v$ is a bipartite graph for which $J_{G\setminus v}$ has a pure resolution, by Proposition~\ref{induced}, since $G\setminus v$ is an induced subgraph of $G$. Let $H$ be the graph depicted in Figure~\ref{square}. Note that, $G\setminus v$ is connected, because by replacing $v$ by $w$ in each path between two vertices $x$ and $y$ of $G$ which contains $v$, one obtains a path between $x$ and $y$ in $G\setminus v$, as otherwise, $G$ is not $H$-free, a contradiction by Remark~\ref{special graphs}. If $G\setminus v$ is a tree, then $G\setminus v$ is $P_{n-1}$ or $K_{1,n-2}$. If $G\setminus v$ is $P_{n-1}$, then, clearly, $G$ has an induced subgraph isomorphic to $H$, containing $C$, which is a contradiction, by Remark~\ref{special graphs}. If $G\setminus v$ is $K_{1,n-2}$, then $|A|=2$. So that, if $G$ is not complete bipartite, in this case, then there is a vertex $x$ in $B$ which is not adjacent to $v$. Therefore, the induced subgraph of $G$ on $\{v,w,z,t,x\}$ is isomorphic to $H$, a contradiction. Now, assume that $G\setminus v$ is not a tree and contains an induced cycle of length $4$. Thus, by the induction hypothesis, we have that $G\setminus v$ is a complete bipartite graph. So that $G$ is also complete bipartite, because otherwise, there is a vertex $u\in B$ which is not adjacent to $v$ and consequently, the induced subgraph on $\{v,w,z,t,u\}$ is isomorphic to $H$, a contradiction, since $J_G$ has a pure resolution. Now, assume that $G$ is a disconnected graph with connected components $G_1,\ldots,G_c$, where $c\geq 2$. By Proposition~\ref{induced}, $J_{G_i}$ has a pure resolution, for all $i=1,\ldots,c$. Hence, $G_i$ is a path, or a complete graph, or a complete bipartite graph, for all $i=1,\ldots,c$, by the above discussion. On the other hand, $J_{G_i\cup G_j}$ has also a pure resolution, for all $1\leq i<j\leq c$, again by Proposition~\ref{induced}. Now, it suffices to apply Lemma~\ref{lemma} for $J_{G_i}$ and $J_{G_j}$. Therefore, all connected components of $G$ are paths, by Theorem~\ref{linear}, Remark~\ref{special graphs}, and \cite[Theorem~5.3]{SZ}. $\Box$

\section{"Free cut edge" switching and the linear strand of binomial edge ideals }\label{Free cut edge switching and linear strand of binomial edge ideals}

\noindent In this section, we introduce a new switching for graphs and study some algebraic invariants of binomial edge ideals through this switching. Also, we study the linear strand of binomial edge ideals via graphical terms.

Let $G$ be a graph on $[n]$ and $v$ a vertex of it. If $e=\{v,w\}$ is an edge of $G$, then two vertices $v$ and $w$ are called the endpoints of $e$. If $\{e_1,\ldots,e_t\}$ is a set of edges of $G$, then by $G\setminus \{e_1,\ldots,e_t\}$, we mean the graph on the same vertex set as $G$ in which the edges $e_1,\ldots,e_t$ are omitted. To simplify our notation, we write $G\setminus e$, instead of $G\setminus \{e\}$.
An edge $e$ of $G$ whose deletion from the graph, yields a graph with more connected components than $G$, is called a \textbf{cut edge} of $G$.

Let $\Delta(G)$ be the clique complex of $G$, the simplicial complex whose facets are the vertex sets of the maximal cliques of $G$. We say that a vertex of $G$ is a \textbf{free vertex}, if it is a free vertex of $\Delta(G)$, i.e. it is contained in only one facet of $\Delta(G)$. Let $G$ be a graph and $e$ a cut edge of $G$ such that its endpoints are the free vertices of the graph $G\setminus e$. Then, we call $e$, a \textbf{free cut edge} of $G$. Suppose that $\{e_1,\ldots,e_t\}$ is the set of all free cut edges of $G$. Then, we call the graph $G\setminus  \{e_1,\ldots,e_t\}$, the \textbf{reduced graph} of $G$, and denote it by $\mathcal{R}(G)$ (see \cite[Definition~3.3]{KS}). Set $\mathcal{R}(G):=G$, if $G$ does not have any free cut edges.

As in \cite{MSh}, if $v,w$ are two vertices of a graph $G=(V,E)$ and $e=\{v,w\}$ is not an edge of $G$, then $G_e$ is defined to be the graph on the vertex set $V$, and the edge set $E \cup \{\{x,y\}~:~x,y\in N_G(v)~\mathrm{or}~x,y\in N_G(w)\}$. The set of all neighbors (i.e. adjacent vertices) of the vertex $v$ in $G$, is denoted by $N_G(v)$. By $f_e$, we mean the binomial $f_{ij}=x_iy_j-x_jy_i$, in which $e=\{i,j\}$ is an edge of $G$. Now, we recall the following propositions from \cite{KS}.

\begin{prop}\label{colon1}
\cite[Proposition~3.7]{KS} Let $G$ be a graph and $e$ be a cut edge of $G$. Then \\
{\em{(a)}} $\beta_{i,j}(J_G)\leq \beta_{i,j}(J_{G\setminus e})+\beta_{i-1,j-2}(J_{{(G\setminus e)}_e})$, for all $i,j\geq 1$, \\
{\em{(b)}} $\mathrm{pd}(J_G)\leq \mathrm{max}\{\mathrm{pd}(J_{G\setminus e}),\mathrm{pd}(J_{{(G\setminus e)}_e})+1\}$, \\
{\em{(c)}} $\mathrm{reg}(J_G)\leq \mathrm{max}\{\mathrm{reg}(J_{G\setminus e}),\mathrm{reg}(J_{{(G\setminus e)}_e})+1\}$.
\end{prop}

\begin{prop}\label{colon2}
\cite[Proposition~3.9]{KS} Let $G$ be a graph and $e$ be a free cut edge of $G$. Then we have \\
{\em{(a)}} $\beta_{i,j}(J_G)=\beta_{i,j}(J_{G\setminus e})+\beta_{i-1,j-2}(J_{G\setminus e})$, for all $i,j\geq 1$, \\
{\em{(b)}} $\mathrm{pd}(J_G)=\mathrm{pd}(J_{G\setminus e})+1$, \\
{\em{(c)}} $\mathrm{reg}(J_G)=\mathrm{reg}(J_{G\setminus e})+1$.
\end{prop}

Motivated by Proposition~\ref{colon2}, we introduce a new method of switching of graphs which is based on the existence of some free cut edges and free vertices in a graph. Let $v,w$ be two vertices of $G$. Then, by $G\cup e$, we mean the graph obtained from $G$, by adding the edge $e=\{v,w\}$ to $G$, where $e$ is not an edge of $G$.

\begin{defn}\label{switching}
{\em Let $G$ be a graph and $e$ be a free cut edge of $G$. Then, we say that $G':=(G\setminus e)\cup e'$ is obtained from $G$, by \textbf{"free cut edge" switching}, if $e'$ is a free cut edge of $G'$. }
\end{defn}

Note that $G$ and $G'$ are on the same vertex set, and also $e$ and $e'$ could have a vertex in common. This switching is possible if and only if $G$ has at least a free cut edge and at least a free vertex. For example, one could do it on a forest and obtain again a forest. \\

Now, let us recall a criterion for checking the closedness of a graph due to Cox and Erskine in \cite{CE}. They call a connected graph $G$, \textit{narrow}, if every vertex is distance at most one from every longest shortest path. Here, the distance $d(v,w)$ between two vertices $v,w$ of $G$ is the length of the shortest path connecting them; the diameter $\mathrm{diam}(G)$ of $G$ is the maximum distance between two vertices
of $G$, and a shortest path connecting two vertices $v$ and $w$ for which $d(v,w)=\mathrm{diam}(G)$, is called a longest shortest path of $G$. Finally, in \cite{CE}, it is shown that $G$ is closed if and only if it is chordal, claw-free and narrow.

\begin{exam}
{\em Let $G$ be the graph depicted in Figure~\ref{Switching1}. The edge $e$ is a free cut edge and the vertex $z$ is a free vertex of $G\setminus e$. We can apply free cut edge switching to $G$. We delete $e$ and join the free vertices $v,w$ of $G\setminus e$, by an edge $e'$ to obtain the graph $G'$, shown in Figure~\ref{Switching2}. The edge $e'$ is also a free cut edge of $G'$. Note that, $\mathrm{diam}(G)=4$ and $\mathrm{diam}(G')=5$. Thus, obviously, $G$ and $G'$ are non-isomorphic graphs, for instance, by the above criterion of Cox and Erskine, we have $G$ is non-closed and $G'$ is closed. Note that, $P:x,v,z,u,y$ is a longest shortest path of $G$, and the distance of $w$ from $P$ is $2$, which implies that $G$ is not narrow nd hence not closed. Using the same criterion, one could easily see that $G'$ is closed. But, as we see below, many properties of them coincide. }
\end{exam}
\begin{center}
\begin{figure}
\hspace{0 cm}
\includegraphics[height=3.1cm,width=4.3cm]{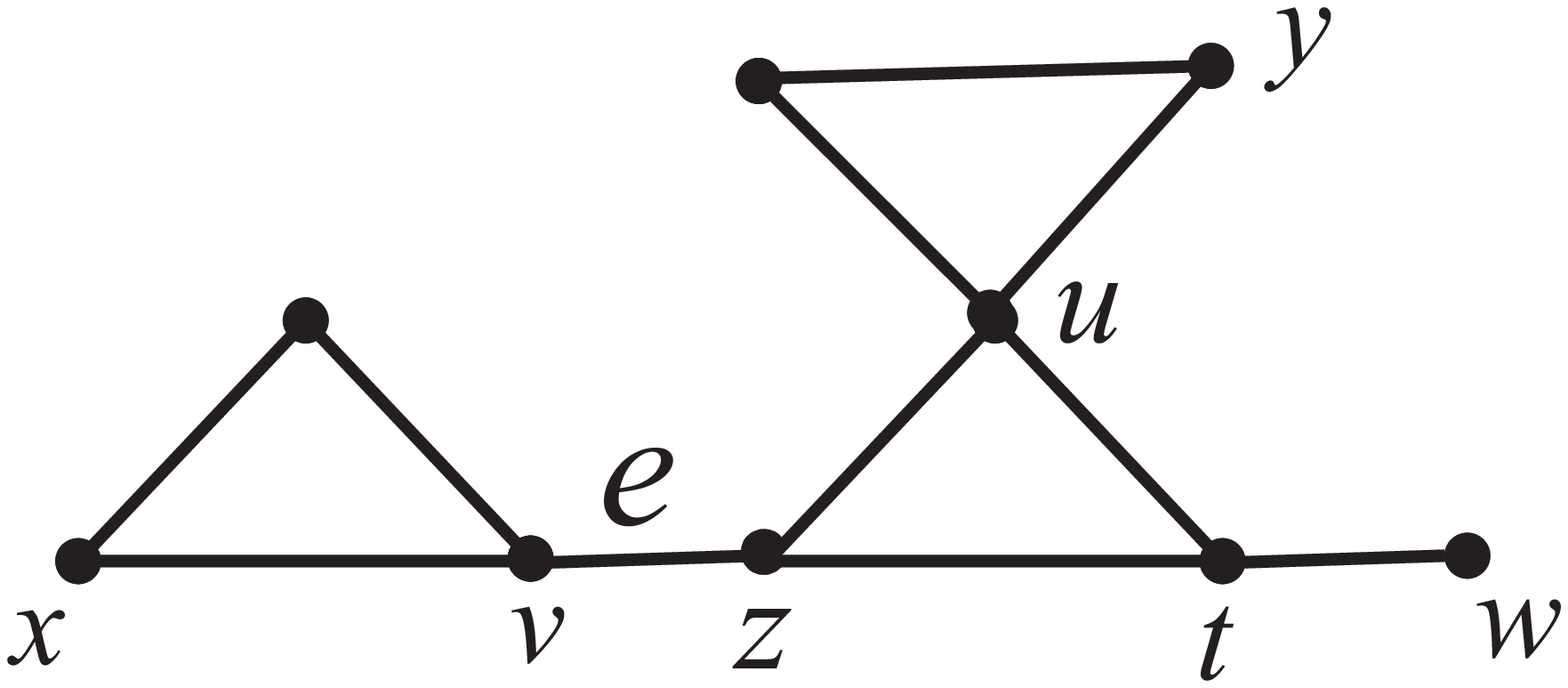}
\caption{\footnotesize{}}\hspace{1.5 cm}
\label{Switching1}
\end{figure}
\end{center}
\begin{center}
\begin{figure}
\hspace{0 cm}
\includegraphics[height=3.1cm,width=4.3cm]{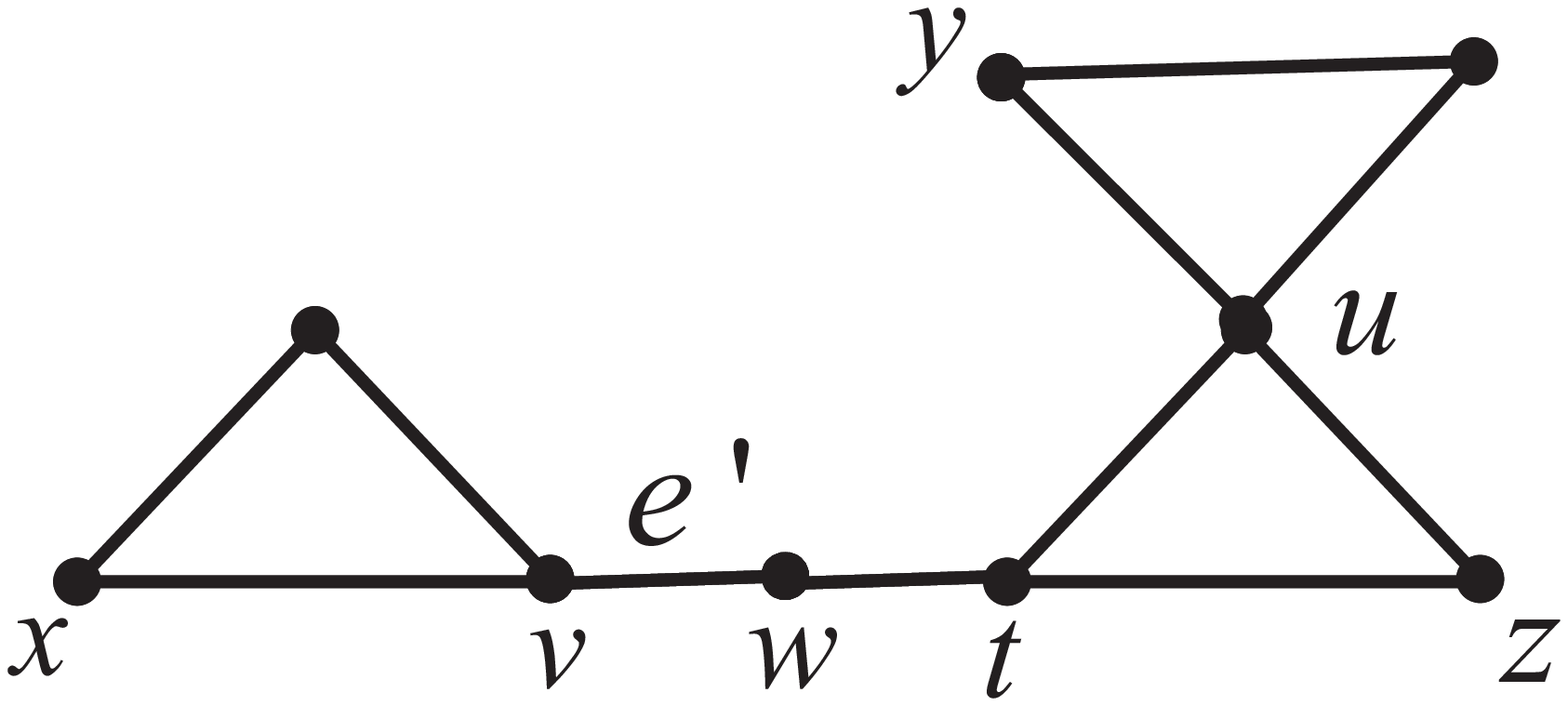}
\caption{\footnotesize{}}\hspace{1.5 cm}
\label{Switching2}
\end{figure}
\end{center}
By the construction of the switching and Proposition~\ref{colon2}, we have:

\begin{prop}\label{switching-betti}
Let $G$ be a graph and $G'$ a graph obtained from $G$ by a "free cut edge" switching. Then \\
{\em{(a)}} $\beta_{i,j}(J_G)=\beta_{i,j}(J_{G'})$, for all $i,j$. \\
{\em{(b)}} $\mathrm{pd}(J_G)=\mathrm{pd}(J_{G'})$. \\
{\em{(c)}} $\mathrm{reg}(J_G)=\mathrm{reg}(J_{G'})$.
\end{prop}

\begin{rem}\label{invariants}
{\em The above proposition shows that the graded Betti numbers, projective dimension and regularity of the binomial edge ideals of graphs are invariant under "free cut edge" switching. Thus, through this operation, we might get some non-isomorphic graphs with different binomial edge ideals, whose minimal graded free resolutions are numerically the same. One could also ask similar questions about some other algebraic or combinatorial properties and invariants, through this switching. }
\end{rem}

\begin{rem}\label{cut edge switching}
{\em One may think of defining a more general method of switching, using "cut edges" instead of "free cut edges". But, the properties mentioned in Proposition~\ref{switching-betti} are not valid in that case anymore. For example, consider $G$ to be the graph shown in Figure~\ref{tree}. The edge $e$ is a cut edge, which is not a free cut edge. If we delete the edge $e$ and add $e':=\{x,y\}$ instead, we gain $P_6$. As we showed in the previous section, $J_{P_6}$ has a pure resolution, but $J_{G}$ does not. }
\end{rem}
\begin{center}
\begin{figure}
\hspace{0 cm}
\includegraphics[height=1.5cm,width=3.1cm]{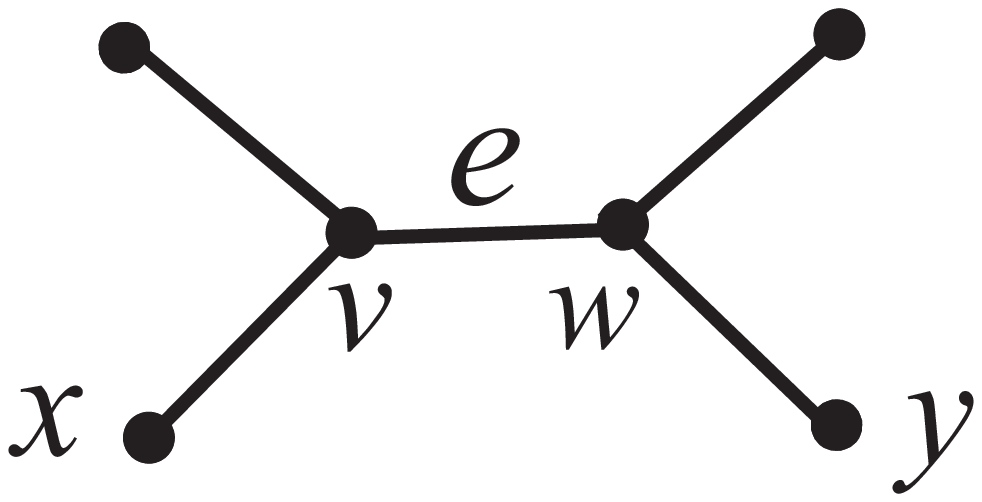}
\caption{\footnotesize{}}\hspace{2 cm}
\label{tree}
\end{figure}
\end{center}
Now, we want to compute the exact values of some of these algebraic invariants of binomial edge ideals of some classes of graphs. \\

Let $m\geq 2$. A \textbf{lollipop} graph, denoted by $L_{m,t}$, is a graph which is obtained from a complete graph $K_m$ and a path $P_t$ such that a vertex of $K_m$ and a leaf of $P_t$ are joined. Figure~\ref{Lollipop} shows a lollipop graph $L_{m,t}$.
\begin{center}
\begin{figure}
\hspace{0 cm}
\includegraphics[height=1.6cm,width=7.5cm]{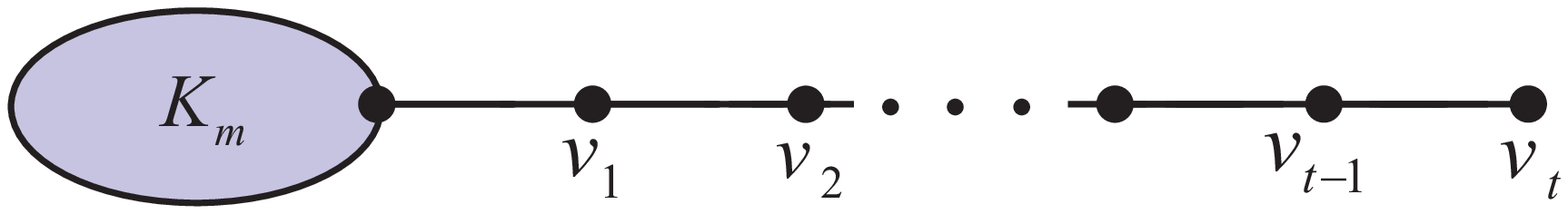}
\caption{\footnotesize{}}\hspace{2 cm}
\label{Lollipop}
\end{figure}
\end{center}
A \textbf{k-handle lollipop} graph, denoted by $L_{m,t_1,\ldots,t_k}$, where $m\geq 2$ and $1\leq k\leq m$, is a graph obtained from a complete graph $K_m$ and $k$ paths of lengths $t_1,\ldots,t_k$, such that a leaf of each path is joined to a vertex of $K_m$, respectively. One could easily see that $L_{m,t_1,\ldots,t_k}$ is obtained by some consecutive "free cut edge" switchings, from $L_{m,t}$, where $t=\sum_{i=1}^k t_i$. Note that the order of $t_1,\ldots,t_k$ is not important in the notation of $L_{m,t_1,\ldots,t_k}$. It is easy to see that $L_{m,t}$ is a closed graph, but $L_{m,t_1,\ldots,t_k}$ might not be closed. Figure~\ref{Lollipop1} shows the lollipop graph $L_{3,3}$, and Figure~\ref{Lollipop2} shows the $2$-handle lollipop graph $L_{3,1,2}$, which is obtained by a "free cut edge" switching from $L_{3,3}$.

One of the benefits of Proposition~\ref{switching-betti} is that while we study such algebraic invariants, we can restrict the problem to a much simpler case. The following proposition could be a good example of this fact. There, we turn the problem to one for a simple closed graph. Note that the result of \cite{ZZ} on the class $\mathcal{G}_3$ of graphs is a special case of the following.

\begin{prop}\label{lollipop}
Let $G=L_{m,t_1,\ldots,t_k}$ be a $k$-handle lollipop graph, where $t=\sum_{i=1}^k t_i$. Then we have
\begin{displaymath}
\beta_{i,i+j}(S/J_G)=\left \{\begin {array}{ll} (i-j+1){m\choose i-j+2}{t\choose j-1}&if~~~
1\leq j\leq i-1\\\\
{t\choose i-1}{m\choose 2}+{t\choose i}&if~~~j=i.
\end{array}\right.
\end{displaymath}
Moreover, $\mathrm{pd}(S/J_G)=m+t-1$ and $\mathrm{reg}(S/J_G)=t+1$.
\end{prop}

\begin{proof}
As mentioned above, it is enough to prove the statements for $G:=L_{m,t}$. Since $G$ is closed, we have $\beta_{i,i+j}=0$, for all $j>i$, by \cite[Theorem~2.2, part (c)]{SK}. We can label the vertices of $G$, such that the cliques of $G$ are intervals $[1,m]$ and $[m+i,m+i+1]$ for $0\leq i\leq t-1$. Now, we use the Betti polynomial $B_{S/J_G}(p,q)=\sum_{i,j}\beta_{i,j}(S/J_G)p^iq^j$. We have $B_{S/\mathrm{in}_{<}J_G}(p,q)=\sum_{i,j}\beta_{i,j}(S/\mathrm{in}_{<}J_G)p^iq^j=B_{S/\mathrm{in}_{<}J_{K_m}}(p,q){(B_{S/\mathrm{in}_{<}J_{K_2}}(p,q))}^t$, since $\mathrm{in}_{<}J_{K_m}$ and other $t$ monomial ideals $\mathrm{in}_{<}J_{K_2}$ associated to the cliques of $G$, are on disjoint sets of variables. By \cite[Theorem~1.1]{EHH}, $S/J_G$ is Cohen-Macaulay. So, by \cite[Proposition~3.2]{EHH}, we have $\beta_{i,j}(S/J_G)=\beta_{i,j}(S/\mathrm{in}_{<}J_G)$, for all $i,j$. So, we get $$B_{S/J_G}(p,q)=B_{S/J_{K_m}}(p,q){(1+pq^2)}^t.$$ Now, we have that $\beta_{i,i+j}(S/J_G)$ is equal to the coefficient of $p^iq^{i+j}$ in the latter polynomial. By an easy computation, we have $$\beta_{i,i+j}(S/J_G)=\sum_{l=0}^t{t\choose l}\beta_{i-l,i+j-2l}(S/J_{K_m}),$$ for all $i,j$. Since the Eagon-Northcott complex minimally resolves $S/J_{K_m}$, we have that for $i\geq 1$, $\beta_{i,i+j}(S/J_{K_m})=i{m\choose i+1}$, if $j=1$, and $\beta_{i,i+j}(S/J_{K_m})=0$, otherwise. Thus, we get $\beta_{i,i+j}(S/J_G)={t\choose j-1}\beta_{i-j+1,i-j+2}(S/J_{K_m})=(i-j+1){m\choose i-j+2}{t\choose j-1}$, for $1\leq j\leq i-1$, and $\beta_{i,2i}(S/J_G)={t\choose i-1}\beta_{1,2}(S/J_{K_m})+{t\choose i}\beta_{0,0}(S/J_{K_m})={t\choose i-1}{m\choose 2}+{t\choose i}$. On the other hand, since $\mathrm{pd}(S/J_{K_m})=m-1$ and $\mathrm{reg}(S/J_{K_m})=1$, the result follows by Proposition~\ref{colon2}.
\end{proof}

\begin{center}
\begin{figure}
\hspace{0 cm}
\includegraphics[height=2.68cm,width=3.45cm]{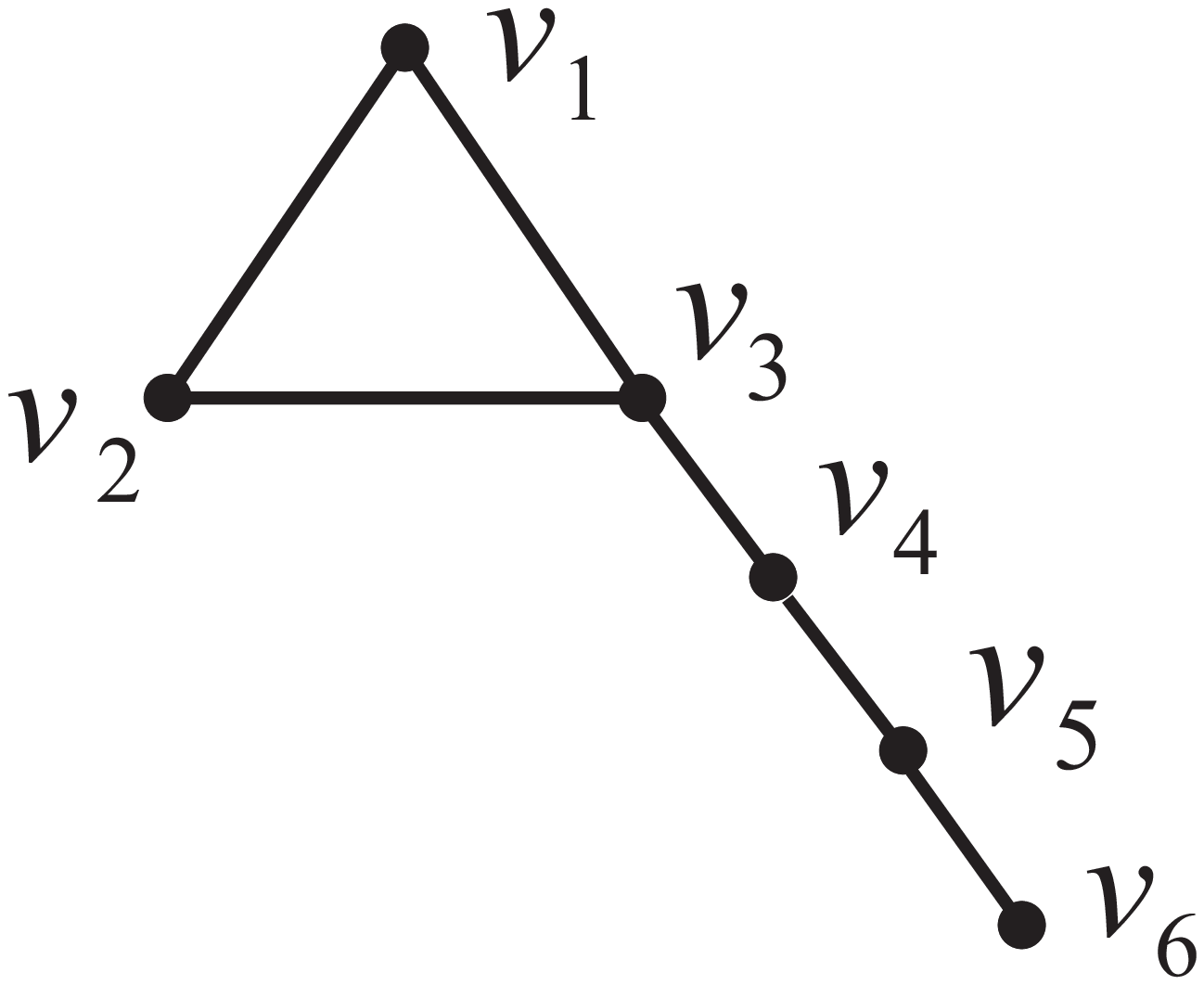}
\caption{\footnotesize{}}\hspace{2 cm}
\label{Lollipop1}
\end{figure}
\end{center}
\begin{center}
\begin{figure}
\hspace{0 cm}
\includegraphics[height=2.89cm,width=3.25cm]{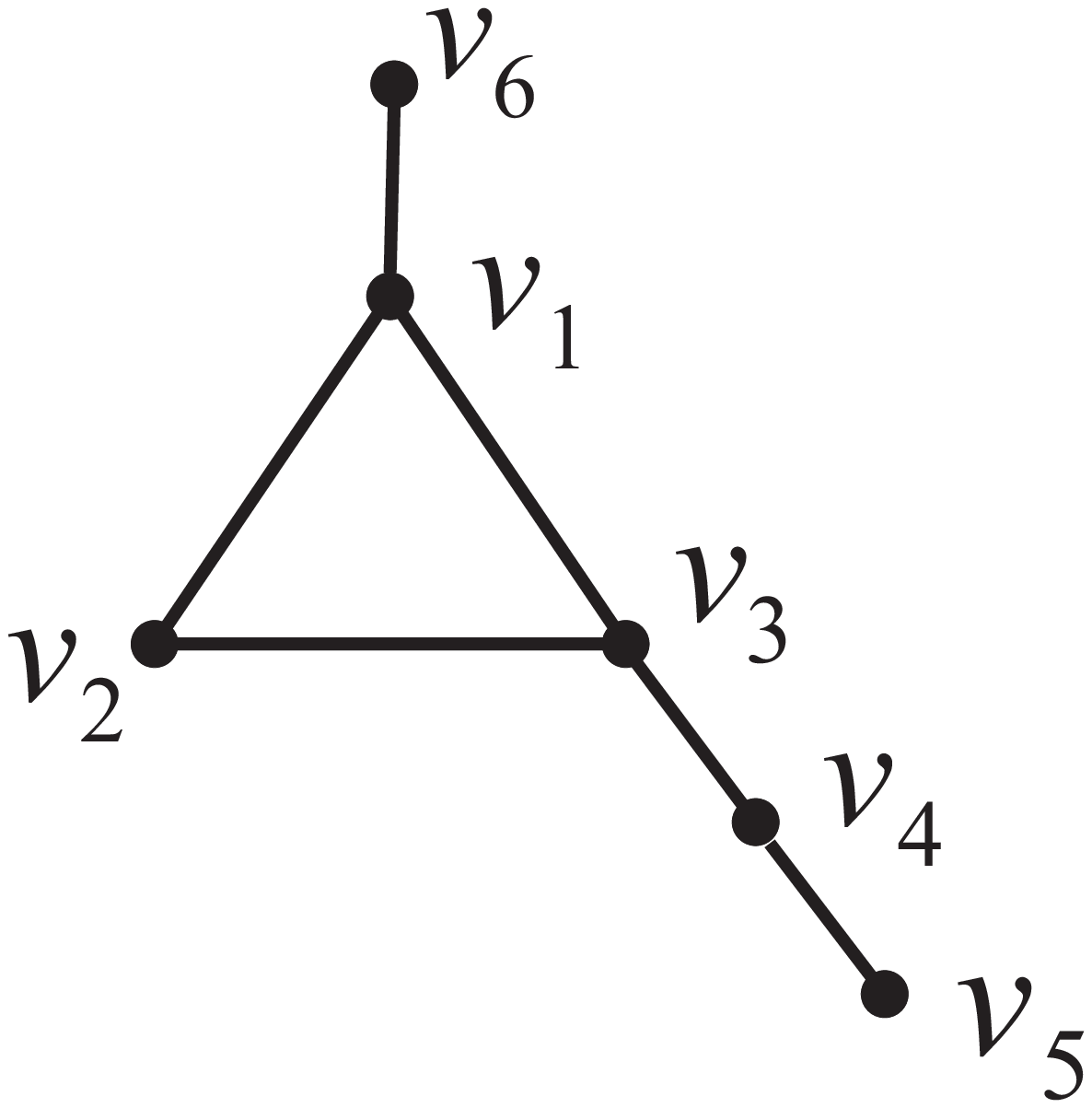}
\caption{\footnotesize{}}\hspace{2 cm}
\label{Lollipop2}
\end{figure}
\end{center}
Let $G$ be a graph. For every $i\geq 1$, denoted by $k_{i}(G)$, we mean the number of cliques of $G$ which are isomorphic to the complete graph $K_{i}$. As an immediate consequence of Proposition~\ref{lollipop}, we have:

\begin{cor}\label{linear strand-lollipop}
Let $G=L_{m,t_1,\ldots,t_k}$ be a $k$-handle lollipop graph. Then we have $\beta_{i,i+2}(J_G)=(i+1)k_{i+2}(G)$.
\end{cor}

\begin{rem}\label{Sharpness2}
{\em By Proposition~\ref{lollipop}, we have $\mathrm{reg}(S/J_G)=c(G)+1$, where $G$ is a $k$-handle lollipop and $c(G)$ is the number of maximal cliques of $G$. So, it gives some closed and non-closed graphs for which the bound claimed in \cite[Conjecture~B]{KS} is the best. So, as we also mentioned in \cite[Remark~3.22]{KS}, if this conjecture could be proved, then the given bound is sharp. }
\end{rem}

Now, we focus on the linear strand of the binomial edge ideals. The following theorem determines the linear strand of the binomial edge ideal of a class of graphs, which is closed under "free cut edge" switching, that is if $G$ belongs to this class, then every (possible) "free cut edge" switching of $G$ also is in this class. It generalizes Corollary~\ref{linear strand-lollipop}.

\begin{thm}\label{linear strand2}
Let $G$ be a graph such that each connected component of $\mathcal{R}(G)$ is either $K_3$-free or a clique. Then $\beta_{i,i+2}(J_G)=(i+1)k_{i+2}(G)$.
\end{thm}

\begin{proof}
Obviously, $\beta_{0,2}(J_G)=k_2(G)$. Thus, assume that $i\geq 1$. Let $R_1,\ldots,R_q$ be the connected components of $\mathcal{R}(G)$, where $q\geq 1$. Since $J_{R_1},\ldots,J_{R_q}$ are on the disjoint sets of variables, the minimal graded free resolution of $S/J_{\mathcal{R}(G)}$ is the tensor product of those of $S/J_{R_j}$, for all $j=1,\ldots,q$. So, by an easy computation, we have $\beta_{i,i+2}(J_{\mathcal{R}(G)})=\sum_{j=1}^q \beta_{i,i+2}(J_{R_j})$. By assumption, each $R_j$ is a clique or a $K_3$-free graph. If $R_j$ is a clique for some $j$, then by the Eagon-Northcott complex, as was mentioned in the proof of Proposition~\ref{lollipop}, we have $\beta_{i,i+2}(J_{R_j})=(i+1)k_{i+2}(R_{j})$. If $R_j$ is $K_3$-free for some $j$, then we have also $\beta_{i,i+2}(J_{R_j})=0$, by \cite[Corollary~2.3]{SK}. Therefore, we get $\beta_{i,i+2}(J_{\mathcal{R}(G)})=(i+1)\sum_{j=1}^q k_{i+2}(R_j)=(i+1)k_{i+2}(\mathcal{R}(G))$. But, $k_{i+2}(G)=k_{i+2}(\mathcal{R}(G))$, because we assumed that $i\geq 1$. So, the statement follows by Proposition~\ref{colon2} (see also \cite[Corollary~3.12]{KS}).
\end{proof}

In Theorem~\ref{linear strand2}, we precisely obtained the linear strand of the binomial edge ideal of some graphs. On the other hand, for every graph $G$, we have $\beta_{1,3}(J_G)=2k_3(G)$, by \cite[Theorem~2.2, part (a)]{SK}. It is an interesting problem to ask about the linear strand of the binomial edge ideal of a graph, in general. By the above results and also some computations with \textit{CoCoA}, it seems that the formula gained in the above cases might be a general formula. So that we pose the following conjecture: \\

\noindent \textbf{Conjecture.} Let $G$ be a graph. Then $\beta_{i,i+2}(J_G)=(i+1)k_{i+2}(G)$. \\

Ene, Herzog and Hibi's conjecture mentioned at the end of \cite{EHH}, together with the following proposition, might strengthen the conjecture in the case of closed graphs. By the notation of \cite{SK} and \cite{SK1}, we use $\mathrm{in}_{<}(G)$ to denote the bipartite graph associated to a closed graph $G$. Also, as usual, we denote the (monomial) edge ideal of a graph $G$, by $I(G)$. Moreover, recall that the linear strand of $I(G)$ is computed as follows:

\begin{prop}\label{Roth-Van Tuyl}
\cite[Proposition~2.1]{RV} Let $G$ be a simple graph. Then $$\beta_{i,i+2}(I(G))=\sum_{W\subseteq V,|W|=i+2}(\sharp \mathrm{comp}(G_{W}^c)-1),$$
where $\sharp \mathrm{comp}(G_{W}^c)$ denotes the number of connected components of the induced subgraph of the complementary graph of $G$ on $W$.
\end{prop}

\begin{thm}\label{linear strand-closed}
Let $G$ be a closed graph. Then $$\beta_{i,i+2}(J_G)\leq \beta_{i,i+2}(\mathrm{in}_{<}J_G)=\beta_{i,i+2}(I(\mathrm{in}_{<}G))=(i+1)k_{i+2}(G).$$
\end{thm}

\begin{proof}
The first inequality is well-known (see for example \cite[Corollary~3.3.3]{HH}). The first equality is based on the definition and notation mentioned in \cite[Section~3]{SK}. So, we should prove that $\beta_{i,i+2}(I(\mathrm{in}_{<}G))=(i+1)k_{i+2}(G)$. For simplicity, set $H:=\mathrm{in}_{<}G$. Note that $H$
is a bipartite graph with the vertex bipartition $V=X\cup Y$, where $X=\{x_1,\ldots,x_n\}$ and $Y=\{y_1,\ldots,y_n\}$. By Proposition~\ref{Roth-Van Tuyl}, one has $$\beta_{i,i+2}(I(H))=\sum_{W\subseteq V,|W|=i+2}(\sharp \mathrm{comp}(H_{W}^c)-1).$$ Let $W\subseteq V$, with $|W|=i+2$. Then $\sharp \mathrm{comp}(H_{W}^c)=1$ or $2$, since $H_{X}$ and $H_{Y}$ are cliques. Let $\sharp \mathrm{comp}(H_{W}^c)=2$. Then $W_1:=W\cap X\neq \emptyset$, $W_2:=W\cap Y\neq \emptyset$ and there is no edge between $W_1$ and $W_2$ . The indexes of all of the vertices of $W_1$ are less than all the vertices of $W_2$, because by the definition of $H$, all the edges of the form $\{x_j,y_l\}$, where $l\leq j$, are in the graph $H^c$. Set $W_1=\{x_{t_1},\ldots,x_{t_k}\}$ and $W_2=\{y_{t_{k+1}},\ldots,y_{t_{i+2}}\}$. Then, all the edges between $W_1$ and $W_2$ occur in $H$. These edges correspond to some edges of the clique of $H$ on $i+2$ vertices $\{v_{t_1},\ldots,v_{t_{i+2}}\}$, where $t_1<\cdots <t_{i+2}$, by using the closedness of $G$. Now, suppose that $C$ is a clique of $G$ on $i+2$ vertices $\{v_{t_1},\ldots,v_{t_{i+2}}\}$, where $t_1<\cdots <t_{i+2}$. Then, it gives exactly $i+1$ distinct subsets of $V$ with the above properties, namely $Z_k=\{x_{t_1},\ldots,x_{t_k},y_{t_{k+1}},\ldots,y_{t_{i+2}}\}$, where $1\leq k\leq i+1$. So that each clique of $H$ which is isomorphic to $K_{i+1}$ corresponds exactly to $i+1$ appropriate subsets of the vertex set of $H$. Thus, we have $\sum_{W\subseteq V,|W|=i+2}(\sharp \mathrm{comp}(H_{W}^c)-1)=(i+1)k_{i+2}(G)$, which implies the result.
\end{proof}

By Theorem~\ref{linear strand-closed} and \cite[Proposition~3.2]{EHH}, we get:

\begin{cor}\label{CM closed}
Let $G$ be a closed graph with Cohen-Macaulay binomial edge ideal. Then $\beta_{i,i+2}(J_G)=(i+1)k_{i+2}(G)$.
\end{cor}

Theorem~\ref{linear strand-closed} might also support our conjecture for some non-closed graphs, as we see below.

\begin{cor}\label{linear strand3}
Let $G$ be a non-closed graph and $e$ be a cut edge of $G$. If $G\setminus e$ is a closed graph, then $\beta_{i,i+2}(J_G)\leq (i+1)k_{i+2}(G)$
\end{cor}

\begin{proof}
Clearly, $\beta_{0,2}(J_G)=k_2(G)$. By Proposition~\ref{colon1} (see also \cite[Corollary~3.12]{KS}), we have $\beta_{i,i+2}(J_G)\leq \beta_{i,i+2}(J_{G\setminus e})$, for all $i$. On the other hand, we have that $k_{i+2}(G)=k_{i+2}(G\setminus e)$, for $i\geq 1$. Now, it suffices to apply Proposition~\ref{linear strand-closed}, since $G\setminus e$ is a closed graph.
\end{proof}

For example, the graph shown in Figure~\ref{Switching1} satisfies the condition of the above proposition. Note that in the above, $e$ is a cut edge, not necessarily a free cut edge. \\

\textbf{Acknowledgments:} The authors would like to thank the referee for suggesting Lemma~\ref{lemma} which shortened the original proof of the main theorem of the paper, and also for his or her other valuable comments. Moreover, the authors would like to thank to the Institute for Research in Fundamental Sciences (IPM) for financial support. The research of the first author was in part supported by a grant from IPM (No. 93050220).

\providecommand{\byame}{\leavevmode\hbox
to3em{\hrulefill}\thinspace}

\end{document}